\documentclass{amsart}

\title{Definable henselian valuations}

\author{Franziska Jahnke}
\address{Institut f\"ur Mathematische Logik\\Einsteinstr. 62\\48149 M\"unster, 
Germany}
\email{franziska.jahnke@wwu.de}

\author[Jochen Koenigsmann]{Jochen Koenigsmann}
\address{Mathematical Institute\\Woodstock Road\\Oxford OX2 6CG, UK}
\email{koenigsmann@maths.ox.ac.uk}

\usepackage{enumerate}
\usepackage{braket}
\usepackage{txfonts}
\newtheorem{Th}{Theorem}[section]
\newtheorem{Thm}[Th]{Theorem}

\newtheorem*{Def}{Definition}
\newtheorem{Cor}[Th]{Corollary}
\newtheorem{Prop}[Th]{Proposition}

\newtheorem*{Ex}{Example}
\newtheorem*{Rem}{Remark}
\newtheorem{Obs}[Th]{Observation}
\newtheorem{Lem}[Th]{Lemma}

\thanks{The results in 
this paper are part of the first author's PhD thesis
under the supervision of the second author.
The research leading to these results has received funding from the 
[European Community's] Seventh Framework Programme 
[FP7/2007-2013] under grant agreement 
number 238381.}


\begin{document}
\begin{abstract}
In this note we investigate the question 
when a henselian valued field
carries a non-trivial $\emptyset$-definable 
henselian valuation (in the language of rings). 
This is clearly not possible when the field is either separably or real closed,
and,
by the work of Prestel and Ziegler, there
are further examples of henselian valued fields which do not admit a
$\emptyset$-definable non-trivial hen\-selian valuation.
We give conditions on the residue field which ensure the
existence of a parameter-free definiton.
In particular, we show that 
a henselian valued field admits a non-trivial henselian 
$\emptyset$-definable valuation 
when the residue field is separably closed or
sufficiently non-henselian, or when the absolute Galois group of the (residue) 
field is non-universal.
\end{abstract}

\maketitle

\section{Introduction}
In a henselian valued field $(K,v)$, many arithmetic or algebraic questions
can be reduced, via the henselian valuation $v$, to simpler questions about
the value group $vK$ and the residue field $Kv$.
By the celebrated Ax-Kochen/Ershov Principle, in fact, 
if the residue characteristic is $0$, `everything' can be so reduced:
the 1st-order theory of $(K,v)$ (as valued field)
is fully determined by the 1st-order theory
of $vK$ (as ordered abelian group) and of $Kv$ (as pure field).
In that sense the valuation (with its two accompanying structures $vK$ and $Kv$)
`knows' everything about $K$, especially the full 1st-order theory of 
$K$ as pure field, or, as one may call it, the {\em arithmetic} of $K$.

Conversely, in all natural examples, and, as we will see, 
in most others as well, a henselian valuation $v$ is so intrinsic to $K$
that it is itself encoded in the arithmetic of $K$,
or, to make this notion precise, that its valuation ring ${\mathcal O}_v$
is 1st-order definable in $K$.
Well known examples are the classical fields
$\mathbb{Q}_p$ and $\mathbb{C}((t))$
with their valuation rings
$$\begin{array}{rcl}
\mathbb{Z}_p & = & \{ x\in\mathbb{Q}_p\mid\exists y\; 1+px^2=y^2\}\mbox{ (for }p\neq 2)\\
\mathbb{C}[[t]] & = & \{ x\in\mathbb{C}((t))\mid\exists y\; 1+tx^2=y^2\}
\end{array}$$
Note that the second example uses the parameter $t$.
This is not necessary though:
one can also find a parameter-free definition
of $\mathbb{C}[[t]]$ in $\mathbb{C}((t))$;
however, as observed in \cite{CDLM13},
it can no longer be an existential definition:
otherwise the definition would go up the tower of isomorphic fields
$$\mathbb{C}((t))\subseteq\mathbb{C}((t^{1/2!}))\subseteq\mathbb{C}((t^{1/3!}))\subseteq\ldots$$
thus leading to a 1st-order definition of a non-trivial valuation subring
of the algebraically closed field
$\mathbb{C}((t^{1/\infty}))=\bigcup_n\mathbb{C}((t^{1/n!}))$,
contradicting quantifier eliminiation
(every definable subset is finite or cofinite).

That $\mathbb{C}[[t]]$ {\em is} $\emptyset$-definable in $\mathbb{C}((t))$
follows from the more general fact
that every henselian valuation
with non-divisible archimedean 
value group is $\emptyset$-definable (\cite{Koe04}).
This has recently been generalized to non-divisible regular value groups
(those elementarily equivalent to archimedean ordered
groups, see \cite{Hon14}). 
Note that there are also several recent preprints which discuss
$\emptyset$-definability of a range of henselian valuations
using only formulae of `simple' quantifier type 
(i.e.\;definitions involving
$\forall$-,$\exists$-,$\forall\exists$ or $\exists\forall$-formulae).
To learn more about 
these exciting developments, we refer the reader to \cite{CDLM13},
\cite{AK13}, \cite{Fe13} and \cite{Pr14}.

In this paper we will develop two new, fairly general criteria, 
one on the residue field and one on the absolute Galois group $G_K$ of $K$
to guarantee $\emptyset$-definability of
(in the first case a given, in the second case, at least some)
henselian valuation on $K$. 
It is well-known that separably and real closed fields admit no
definable henselian valuations. Furthermore,
by the work of Prestel and Ziegler
(\cite{PZ78}, \S 7)
there are henselian valued fields which are neither separably nor real closed 
and which do not admit any
$\emptyset$-definable henselian valuation. 
It is thus a natural question to ask which conditions on a henselian
valued field $(K,v)$ ensure that $v$ is $\emptyset$-definable or that
$K$ admits at least some $\emptyset$-definable henselian valuation. 
In the present work, we focus on parameter-free definitions as a
definition of a henselian valuation 
with parameters need not ensure the existence of a definable 
henselian valuation in elementarily equivalent fields.
Note that there are also examples of
henselian valuations which are not even definable with parameters (see
\cite{DF96}, Theorem 4.4). The only known examples of henselian fields which
admit no parameter-definable henselian valuations at all are 
separably and real closed fields. 

The paper is organized as follows. In the next section, we 
discuss the main tools which we require. We recall the definition of 
$p$-henselian valuations and the 
canonical ($p$)-henselian valuation. Building on work of the second author
(see \cite{Koe95}), the authors 
have shown that 
the canonical $p$-henselian valuation $v_K^p$ is
typically definable (Theorem 3.1 in \cite{JK14}). 
We show that it is furthermore
henselian iff it is coarser than the canonical henselian valuation.

The third section contains the main results of this paper. We begin by giving
conditions on the residue field to make a henselian 
valuation definable.
The first criterion says that the henselian valuation $v$ on $K$
is $\emptyset$-definable
if, for some prime $p$,
$Kv$ allows a separable extension $L$ with $L \neq L(p)$
that does not allow a $p$-henselian valuation
(Theorem \ref{def1}, cf.\;section 2 for the definition of $L(p)$ and
 $p$-henselian).
We deduce from this that any henselian valuation with
finitely generated, hilbertian, PAC or 
simple but not separably closed residue field is
$\emptyset$-definable.
We use a similar method to show 
that a henselian valued field $(K,v)$
where $Kv$ is separably or real closed, but $K$ isn't,
admits some $\emptyset$-definable henselian valuation.

The next part discusses a second, Galois-theoretic criterion
for the existence of a $\emptyset$-definable henselian valuation
on a (non-separably- and non-real-closed) henselian valued field $K$
(Theorem \ref{non-universal}).
It says that if $K$ is henselian and $G_K$ is {\em non-universal},
that is, that not every finite group is a subquotient of $G_K$,
then $K$ admits some $\emptyset$-definable non-trivial henselian valuation.
In most cases, we will in fact define the canonical henselian valuation on $K$.
This generalizes old results by Neukirch, Geyer and Pop
on henselian fields with prosolvable $G_K$.
One class of examples is given by henselian
NIP fields of positive characteristic.

These two criteria,
one on the residue field of a given henselian valuation $v$ on $K$,
and one on $G_K$ in the presence of {\em some} henselian valuation on $K$,
are fairly independent.
One easily finds examples of the first kind where $G_K$ {\em is} universal
and examples where it isn't.
Similarly, there are henselian fields $K$ with non-universal $G_K$
where every henselian valuation on $K$ satisfies the criterion on the 
residue field and such where none of them does.
What is common between the two criteria, however,
is the method of proof which in either case depends on a careful analysis
when, on a field $K$, the canonical $p$-henselian valuation $v_K^p$
is already henselian.
Although many fields have universal absolute Galois groups, 
the best known ones are hilbertian fields and PAC fields
with non-abelian free absolute Galois group. Hence some of the main examples of 
henselian valued fields
for which the second criterion fails are covered by the first one.

\section{Henselian and $p$-henselian valuations}
\subsection{The canonical henselian valuation}
We call a field $K$ \emph{henselian} if it admits some non-trivial henselian
valuation. 
For any field $K$, there is
a \emph{canonical henselian valuation} on $K$. In this section, we
recall the definition and discuss some of its properties.
We use the following notation: For a valued field $(K,v)$,
we denote the valuation ring by ${\mathcal O}_v$, the residue field by $Kv$,
the value group by $vK$ and the maximal ideal by ${\mathfrak m}_v$. 
For an element $a \in {\mathcal O}_v$, we write 
$\overline{a}$ to refer to its image in $Kv$.

\begin{Th}[\'a la F.K. Schmidt] \label{fk}
If a field admits two independent non-trivial hen\-se\-lian valuations, 
then it is
separably closed.
\end{Th}
\begin{proof} \cite{EP05}, Theorem 4.4.1.
\end{proof}

One can deduce from this that 
the henselian valuations on a field
form a tree: Divide the class of henselian valuations on $K$
 into
two subclasses, namely
$$H_1(K) = \Set{v \textrm{ henselian on } K | Kv \neq Kv^{sep} }$$
and
$$H_2(K) = \Set{ v \textrm{ henselian on } K | Kv = Kv^{sep} }.$$
A corollary of the above theorem is that any valuation $v_2 \in H_2(K)$ 
is \emph{finer} than any $v_1 \in H_1(K)$, i.e.  
${\mathcal O}_{v_2} \subsetneq {\mathcal O}_{v_1}$,
and that any two valuations in $H_1(K)$ are comparable.
Furthermore, if $H_2(K)$ is non-empty, then there exists a unique coarsest
$v_K \in H_2(K)$; otherwise there exists a unique finest $v_K \in H_1(K)$.
In either case, $v_K$ is called the \emph{canonical henselian valuation}.
Note that if $K$ is not separably closed and
admits a non-trivial henselian valuation, then $v_K$
is also non-trivial.

As we will usually define henselian valuations 
on finite Galois extensions later on, we often use the fact that 
coarsenings of $v_K$ remain henselian when restricted 
to subfields of finite index:
\begin{Th}[\cite{EP05}, Theorem 4.4.4] 
\label{down}
Let $(L,w)$ be a valued field, and assume that $L$ is not separably closed and 
that $w$ is a (not necessarily proper)
coarsening of $v_L$. If $K \subset L$ is a subfield such that $L/K$ is finite, 
then $v=w|_K$ is a coarsening of $v_K$.
\end{Th}

\subsection{p-henselianity}
Throughout this section, let $K$ be a field and $p$ a prime.
\begin{Def}
We define $K(p)$ to be the compositum of all Galois extensions of $K$ of 
$p$-power degree. 
A valuation $v$ on $K$ is called
\emph{$p$-henselian} if $v$ extends uniquely to $K(p)$.
We call $K$ \emph{$p$-henselian} if $K$ 
admits a non-trivial
$p$-henselian valuation. 
\end{Def}
Clearly, this definition only imposes a condition on $v$ if $K$ admits
Galois extensions
of $p$-power degree. 

\begin{Prop}[\cite{Koe95}, Propositions 1.2 and 1.3] \label{phenseq}
For a valued field $(K,v)$, the following are equivalent:
\begin{enumerate}
\item $v$ is $p$-henselian,
\item $v$ extends uniquely to every Galois extension of $K$ of $p$-power
degree,
\item $v$ extends uniquely to every Galois extension of $K$ of degree
$p$,
\item for every polynomial $f \in {\mathcal O}_v$ which splits in $K(p)$ and
every $a \in {\mathcal O}_v$ with $\bar{f}(\overline{a}) = 0$ and
$\bar{f'}(\overline{a}) \neq 0$, there exists $\alpha
\in {\mathcal O}_v$ with $f(\alpha)=0$ and $\overline{\alpha}=\overline{a}$.
\end{enumerate}
\end{Prop}

As for fields carrying a henselian valuation, there is again a canonical
$p$-henselian valuation, due to the following analogue of Theorem \ref{fk}:

\begin{Thm}[\cite{Br76}, Corollary 1.5]
If $K$ carries two independent non-trivial 
$p$-hen\-se\-lian valuations, then $K = K(p)$.
\end{Thm}

We again divide the class of 
$p$-henselian valuations on $K$
 into
two subclasses,
$$H^p_1(K) = \Set{v\; p\textrm{-henselian on } K | Kv \neq Kv(p)}$$
and
$$H^p_2(K) = \Set{ v\; p\textrm{-henselian on } K | Kv = Kv(p) }.$$
As before, 
one can deduce that any valuation $v_2 \in H^p_2(K)$ 
is \emph{finer} than any $v_1 \in H^p_1(K)$, i.e. 
${\mathcal O}_{v_2} \subsetneq {\mathcal O}_{v_1}$,
and that any two valuations in $H^p_1(K)$ are comparable.
Furthermore, if $H^p_2(K)$ is non-empty, then there exists a unique coarsest
valuation
$v_K^p$ in $H^p_2(K)$; otherwise there exists a unique finest 
valuation $v_K^p \in H^p_1(K)$.
In either case, $v_K^p$ is called the \emph{canonical $p$-henselian valuation}.
Again, if $K$ is $p$-henselian and $K \neq K(p)$ holds, 
then $v_K^p$ is also non-trivial.

Note that unlike henselianity, being $p$-henselian does not go up
arbitrary algebraic extensions, as a superfield might have far more
extensions of $p$-power degree.
Nevertheless, similar to Theorem \ref{down}, sometimes 
$p$-henselianity goes down:
\begin{Prop} 
\label{finite} Let $K$ be a field, $K\neq K(p)$.
Assume that $L$ is a normal algebraic
extension of $K$, where $L$ is 
$p$-henselian and $L\neq L(p)$.
If
\begin{enumerate}
\item $K \subseteq L \subsetneq K(p)$ or
\item $L/K$ is finite
\end{enumerate}
then $K$ is $p$-henselian.
\end{Prop}
\begin{proof} 1.: See \cite{Koe03}, Proposition 2.10.\\
2.: Assume $K$ is not $p$-henselian, and let $v$ be a valuation on $K$.
By the first part of the proposition, $v$ has infinitely many extensions to 
$K(p)$: If there were only $n$ extensions of $v$ to $K(p)$, then there
would be some $L' \supset K$ finite, $L' \subsetneq K(p)$, such that
$v$ had $n$ extensions to $L'$. The normal hull of $L'$ and thus $K$ would be 
$p$-henselian.\\
Now assume $L=K(a_1,\dotsc, a_m)$ finite and normal, then 
$K(p)(a_1, \dotsc, a_m) \subseteq L(p)$. If $w$ is a valuation on $L$, 
then $v = w|_K$ has infinitely many prolongations to $K(p)$.
As $v$ has only finitely many prolongations to $L$, and all these
are conjugate, $w$ must have infinitely many prolongations 
to $K(p)(a_1, \dotsc, a_m)$ and hence to $L(p)$.
\end{proof}

For any valued field, $p$-extensions of the residue field lift to $p$-extensions
of the field.
\begin{Prop}[\cite{EP05}, Theorem 4.2.6] \label{pex}
Let $(K,v)$ be a valued field and $p$ a prime. If $Kv \neq Kv(p)$, 
then $K \neq K(p)$.
\end{Prop}

\subsection{Defining $p$-henselian valuations}
In this section, we recall a Corollary of
the Main Theorem in \cite{JK14} which is used 
in all of our proofs in
later sections.

When it comes to henselian valued fields, real closed fields always
play a special role. By o-minimality, no real closed field admits a definable
henselian valuation, and there are real closed fields
which admit no henselian valuations (like $\mathbb{R}$) whereas others
do (like $\mathbb{R}((t^\mathbb{Q}))$). These difficulties are reflected 
by $2$-henselian valuations on Euclidean fields.
A field $K$ is called \emph{Euclidean} if $[K(2):K] = 2$. 
Any Euclidean field
is uniquely ordered, the positive elements being exactly the squares.
If a Euclidean field has no odd-degree extensions, then it is real closed.
In particular, there is an ${\mathcal L}_\textrm{ring}$-sentence
$\rho$ such that any field $K$ with $K\neq K(2)$ 
models $\rho$ iff it is non-Euclidean. Note that Euclidean fields
are the only fields for which
$K(p)$ can be a finite proper extension of $K$.

\begin{Thm}[Corollary 3.3 in \cite{JK14}]
\label{pdef} \label{def} 
Let $p$ be a prime and 
consider the class of fields
$$\mathcal{K} = \Set{ K | K \;p\textrm{-henselian, with }\zeta_p \in K
\textrm{ in case } \mathrm{char}(K)\neq p}$$
There is a parameter-free $\mathcal{L}_\textrm{ring}$-formula $\phi_p(x)$
such that 
\begin{enumerate}
\item if $p \neq 2$ or $Kv_2$ is not Euclidean, then $\phi_p(x)$ defines 
the valuation ring of the canonical
$p$-henselian valuation $v_K^p$, and
\item if $p=2$ and $Kv_2$ is Euclidean, then $\phi_p(x)$ defines the valuation
ring of the
coarsest $2$-henselian valuation $v_K^{2*}$ such that
$Kv_K^{2*}$ is Euclidean.
\end{enumerate}
\end{Thm}

The existence of such a uniform definition of the canonical $p$-henselian
makes sure that the different cases split into elementary classes:
\begin{Cor} \label{uni}
The classes of fields
$$\mathcal{K}_1 = \Set{K | K \;p\textrm{-henselian, with }\zeta_p \in K
\textrm{ in case } \mathrm{char}(K)\neq p \textrm{ and }v_K^p \in H_1^p(K)}$$
and
$$\mathcal{K}_2 =\Set{K | K \;p \textrm{-henselian, with }\zeta_p \in K
\textrm{ in case } \mathrm{char}(K)\neq p \textrm{ and }v_K^p \in H_2^p(K)}$$
are elementary classes in $\mathcal{L}_\textrm{ring}$.
\end{Cor}
\begin{proof}
The class 
$$\Set{K | K \;p\textrm{-henselian, with }\zeta_p \in K
\textrm{ in case } \mathrm{char}(K)\neq p}$$
is an elementary class in $\mathcal{L}_\textrm{ring}$ by Corollary 2.2 in 
\cite{Koe95}.
The sentence dividing the class into the two elementary subclasses is the 
statement whether the residue field of the valuation defined by $\phi_p(x)$
as in Theorem \ref{def} admits a Galois extension of degree $p$.
Note that if $p=2$ and $Kv_2$ is Euclidean, both $v_K^2$ and 
$v_K^2*$ are elements of $H_1^p(K)$.
\end{proof}

\begin{Rem}
When one is only interested in defining henselian valuations, one can usually
avoid to consider the special case of a Euclidean residue field:
If $(K,v)$ is a henselian valued field, $K$ not real closed and
$Kv$ Euclidean, then -- similarly to Proposition \ref{pex} -- $K$ is also real, 
so $i \notin K$. 
Now $K(i)$ is a $\emptyset$-interpretable extension of $K$, and the unique
prolongation $w$ of $v$ to $K(i)$ has a non-Euclidean residue field, 
namely $Kv(i)$. Thus, in order to get a parameter-free definition of $v$, 
it suffices to define $w$ without parameters on $K(i)$.

However, the same argument does not work for $p$-henselian valuations, as
there is no strong enough analogue of Theorem \ref{down}. 
Thus, for completeness' sake,
we give Theorem \ref{pdef} in its full generality.
\end{Rem}

\subsection{$p$-henselian valuations as henselian valuations}
Let $K$ be a henselian field and $p$ a prime such that $K\neq K(p)$ holds. 
As any henselian valuation is in particular $p$-henselian, we
have either $v_K^p \supseteq v_K$ or $v_K^p \subsetneq v_K$.
In the first case, $v_K^p$ is henselian. As we will make use of this fact
several times later, we note here that this is in fact an equivalence:

\begin{Obs} \label{coarse}
Let $K$ be a henselian field with $K\neq K(p)$ for some prime $p$.
Then
$v_K^p$ is henselian iff $v_K^p$ coarsens $v_K$.
\end{Obs}
\begin{proof} Any coarsening of a henselian valuation -- like $v_K$ -- 
is henselian.
Conversely, assume that $v_K^p$ is henselian and a proper refinement of $v_K$.
Then, by the definition of $v_K$, we get $v_K^p \in H_2(K)$ and
hence $v_K \in H_2(K)$. In this case, $v_K^p$ has a proper coarsening
with $p$-closed residue field, contradicting the definition of
$v_K^p$.
\end{proof}

\section{Main results}
\subsection{Conditions on the residue field}
We first want to show that we can use the canonical $p$-henselian
valuation to define any henselian valuation 
which has not $p$-henselian residue field. 
\begin{Prop} \label{notpdef}
Let $(K,v)$ be a non-trivially
henselian valued field and $p$ a prime. Assume that the residue field
$Kv$ is not $p$-henselian and that $Kv\neq Kv(p)$. 
If $p=2$, assume further that $Kv$ is not Euclidean.
Then $v$ is $\emptyset$-definable.
\end{Prop}
\begin{proof}  
Let $p$ and $(K,v)$ be as above. 
If $\mathrm{char}(K)\neq p$, we assume $\zeta_p \in K$ for now. 

Note that $K\neq K(p)$ (Proposition \ref{pex}). 
Thus, $K$ is $p$-henselian. We claim that $v_K^p = v$.
As $v$ is henselian, it is in particular 
$p$-henselian and
hence comparable to $v_K^p$.
Since $Kv$ is not $p$-henselian, $v_K^p$ is a coarsening of $v$, as 
otherwise $v_K^p$ would induce a $p$-henselian valuation on $Kv$ 
(\cite{EP05}, Corollary 4.2.7). 
Assume $v_K^p$ is a proper coarsening of $v$. Then we get $v \in H_2^p(K)$ 
and hence
$Kv = Kv(p)$, contradicting our assumption on $Kv$. This proves the claim.

For $p=2$, we get from our assumption that
$Kv_K^2=Kv$ is not Euclidean. Thus, $v_K^p$ 
is henselian and $\emptyset$-definable by Theorem \ref{pdef}. 

In case $\mathrm{char}(K)\neq p$ and 
$K$ does not contain a primitive $p$th root of unity, we consider
$K'=K(\zeta_p)$. As this is a $\emptyset$-definable extension of $K$, 
it suffices
to define the 
-- by henselianity unique -- prolongation $v'$ of $v$ to $K'$.
Since $K'v'$ is a finite normal extension of $Kv$ of degree at most $p-1$, 
it still
satisfies $K'v' \neq K'v'(p)$ and is furthermore not $p$-henselian
by Proposition \ref{finite}.
Now $v'$ is $\emptyset$-definable as above, and thus so is $v$.
\end{proof}

Morally speaking, the proposition says that if we have a henselian valued
field $(K,v)$ such that the residue field is `far away' from being henselian,
then $v$ is $\emptyset$-definable. Hence we will now consider well-known 
classes of examples of non-henselian fields and prove that any henselian
valuation with such a residue field is $\emptyset$-definable.

\begin{Ex} Let $k$ be a finite field. Then $G_k \cong \hat{\mathbb{Z}}$, 
in particular $k\neq k(p)$ holds
for all primes $p$. Note that $k$ is not Euclidean since
$\mathrm{char}(k) >0$. As $k$ admits no non-trivial valuations,
$k$ is also not $p$-henselian.
Now by Proposition \ref{notpdef}, 
if $(K,v)$ is a non-trivially henselian valued 
field with $Kv=k$,
then $v$ is $\emptyset$-definable.
\end{Ex}

Probably the best known example of a non-henselian field are the
rationals. One way of showing that the rationals admit no non-trivial henselian
valuation is via Hilbert's Irreducibility Theorem: No hilbertian field is
henselian (see Lemma 15.5.4 in \cite{FJ08}). 
We will now show by a similar proof that furthermore 
any henselian valued field with 
hilbertian residue field
satisfies the assumption of the
above proposition. 
First, we recall the definition of hilbertianity.
\begin{Def}
Let $K$ be a field and let $T$ and $X$ be variables.
Then $K$ is called \emph{hilbertian} if 
for every polynomial $f \in K[T,X]$ which is
separable, irreducible and monic when considered as a polynomial in $K(T)[X]$ 
there is some $a \in K$
such that $f(a,X)$ is irreducible in $K[X]$.
\end{Def}

Note that \emph{Hilbert's Irreducibility Theorem}
 states that $\mathbb{Q}$ is hilbertian.

Examples of hilbertian fields 
include all infinite finitely generated fields, in particular number
fields and function fields over finite fields.

\begin{Lem} \label{hilp}
If $K$ is a hilbertian field then $K \neq K(p)$ for any prime $p$.
Furthermore, $K$ is neither Euclidean nor $p$-henselian.
\end{Lem}
\begin{proof} If $K$ is hilbertian, then $K$ is not Euclidean and 
$K\neq K(p)$ holds for any prime $p$
by Corollary 16.3.6 in \cite{FJ08}.
Let us first treat the case $\mathrm{char}(K) \neq p$.
We may then assume that $K$ contains a primitive $p$th root of unity
as $K(\zeta_p)$ is again hilbertian, and if $K(\zeta_p)$ was $p$-henselian
then so would be $K$ by Proposition \ref{finite}.

Let $v$ be a non-trivial valuation on $K$.
Choose $m \in {\mathfrak m}_v \setminus\{0\}$ 
and consider the irreducible polynomial
$f(T,X)= X^p +mT-1$ in $K(T)[X]$.
If $K$ is hilbertian, there exists an $a \in K^\times$ such that
$f(a,X)$ is irreducible in $K[X]$. Furthermore, 
by exercise 13.4 in \cite{FJ08}, $a$ may be chosen in $\mathcal{O}_v$.
But now $f(a,X)$ splits in $K(p)$, and has a simple zero in $Kv$. Hence
by Proposition \ref{phenseq}, $v$ cannot be $p$-henselian.

In case $\mathrm{char}(K)=p$,
the same argument as above applies to the polynomial
$f(T,X)=X^p+X+mT-2$.
\end{proof}

Combining Theorem \ref{pdef} with Lemma \ref{hilp}, we also get:
\begin{Cor}
Let $(K,v)$ be a henselian valued field such that $Kv$ is hilbertian.
Then $v$ is $\emptyset$-definable.
\end{Cor}

\begin{Ex}
For any number field $K$ and any ordered abelian group $\Gamma$, the power
series valuation on $K((\Gamma))$ is $\emptyset$-definable. 
\end{Ex}

Another well-known class of fields which are not henselian are non-separably
closed PAC fields. As in general -- unlike hilbertian fields -- PAC fields 
do not need to admit any Galois extensions
of prime degree, we give a suitable generalization of Proposition 
\ref{notpdef}. 
Any
non-separably closed PAC field has a finite Galois extension which is still
PAC and which admits in turn Galois extensions of prime degree.
This motivates the following
\begin{Def}
Let $K$ be a field. We call $K$
\emph{virtually not $p$-henselian} if $p \mid \#G_K$ 
and there is some finite Galois extension
$L$ of $K$ with $L \neq L(p)$ such that $L$ is not $p$-henselian.
\end{Def}
Note that if $K \neq K(p)$, then $K$ is virtually not $p$-henselian iff it is
not
$p$-henselian by Proposition \ref{finite}.
We will now show a PAC field $K$ is virtually not $p$-henselian
for any prime $p$ with $p \mid \#G_K$. 
First, we show that a
PAC field $K$ with $K \neq K(p)$ is not $p$-henselian
using the same method as one uses to show that such a field is 
not henselian (see
\cite{FJ08}, Corollary 11.5.5).

\begin{Lem}[Kaplansky-Krasner for $p$-henselian valuations]
Assume that $(K,v)$ is a $p$-hen\-selian valued field
and take $f \in K[X]$ separable,
$\deg(f)>1$, such
that $f$ splits in $K(p)$.
Suppose for each $\gamma \in vK$ there exists some $x \in K$ such that
$v(f(x)) > \gamma$. Then $f$ has a zero in $K$. 
\label{KrasKap}
\end{Lem}
\begin{proof} 
Without loss of generality we may assume that $f$ is monic and that
$\deg(f)= n > 0$. Write 
$$f(X) = \prod\limits_{i=1}^{n} (X-x_i)$$
for $x_i \in K(p)$. Take
$\gamma > n \cdot \max\{v(x_i-x_j) \mid 1 \leq i < j \leq n\}$ and choose
$x \in K$ such that 
$$v(f(x)) = \sum\limits_{i=1}^{n} v(x-x_i) > \gamma.$$
Hence for some $j$ with $1 \leq j \leq n$ we get $v(x-x_j) > \gamma/n$. 
If $x_j \notin K$, then there is some $\sigma \in \mathrm{Gal}(K(p)/K)$
such that $\sigma(x_j) \neq x_j$. Thus, we get
$$v(x-\sigma(x_j))= v(\sigma(x - x_j)) = v(x-x_j) > \dfrac{\gamma}{n},$$
where the last equality holds as $v$ is $p$-henselian. Therefore
$$v(x_j - \sigma(x_j)) \geq \min\{v(x_j-x), v(x - \sigma(x_j))\} > 
\dfrac{\gamma}{n} 
$$
which contradicts the choice of $\gamma$. Hence we conclude
$x_j \in K$, so $f$ has a zero in $K$.
\end{proof}

\begin{Lem} Let $K$ be a field and $p$ a prime. \label{PAC}
If $K$ is PAC and $p$-henselian, then we have $K=K(p)$.
\end{Lem}
\begin{proof} Assume that $K$ is PAC and $p$-henselian. We show that 
$K=K(p)$ holds. 
Take $f \in K[X]$ a separable, irreducible polynomial
with $\deg(f) >1$ splitting in $K(p)$.
It suffices to
show that for all $c \in K^\times$ there exists an $x \in K$ such that
$v(f(x)) \geq v(c)$, as then $f$ has a zero in $K$.

Consider the curve $g(X,Y)=f(X)f(Y)-c^2$. Consider $g(X,Y)$ as a polynomial
over $K^{sep}[Y]$. Eisenstein's criterion 
(\cite{FJ08}, Lemma 2.3.10(b)) applies over this ring
to any linear factor of $f(Y)$, thus $g(X,Y)$ is absolutely irreducible.
As $K$ is PAC, there exist $x, y \in K$ such that $f(x)f(y)=c^2$.
Thus, either $v(f(x))\geq v(c)$ or $v(f(y)) \geq v(c)$ holds.
\end{proof}

As being PAC passes up to algebraic extensions,
any PAC field $K$ is in particular not virtually $p$-henselian
for all primes $p \mid \#G_K$. Furthermore, as real closed fields are not
PAC, no PAC field is Euclidean.

We now give a stronger version of Proposition \ref{notpdef}. The main difference
is that the we drop the assumption on the residue field to admit a Galois 
extension of $p$-power degree for some prime $p$.
\begin{Thm} \label{def1}
Let $(K,v)$ be a non-trivially henselian valued field with 
$p \mid \#G_{Kv}$, and
if $p=2$ assume that $Kv$ is not Euclidean.
If $Kv$ is virtually not $p$-henselian 
then $v$ is $\emptyset$-definable on $K$.
\end{Thm}
\begin{proof} If $Kv$ is virtually not $p$-henselian and 
$Kv \neq Kv(p)$, then $v$ 
is $\emptyset$-definable by Proposition \ref{notpdef}.

In case $Kv = Kv(p)$, by assumption there is a $p$-henselian finite
Galois extension $L$ of $Kv$ with
$L \neq L(p)$. 
As $Kv$ is not Euclidean,
$L$ is also not Euclidean. By Proposition \ref{finite}, we may assume
that $L$ contains a primitive $p$th root of unity in case 
$\mathrm{char}(Kv) \neq p$. Let $[L:Kv]=n$.

Consider any finite Galois extension $M$ of $K$, with $w$ the unique 
prolongation of $v$ to $M$ such that $Mw=L$ holds.
As before, $w$ is $\emptyset$-definable on $M$
(since $w = v_M^p$ as in the proof of Proposition \ref{notpdef}) and
hence, by interpreting $M$ in $K$ using parameters,
so is its restriction $v$ to $K$.

Thus, it remains to show that a definition can be found without parameters.
The interpretation of Galois extensions of a fixed degree of $K$ can be done
uniformly with respect to the parameters (namely the coefficients of a
minimal polynomial generating the extension). By Theorem \ref{pdef}, 
the
definition of the $p$-henselian valuations on these 
can also be done
uniformly. To make sure that the residue field of the canonical
$p$-henselian valuation of a finite Galois extension of $K$ corresponds
to a field $L$ as described above, we need to restrict to extensions
$M$ of $K$ with $v_p^M \in H_1^p(M)$. 
By Corollary \ref{uni}, 
this is a $\emptyset$-definable condition.
Hence we get
the desired definition by
\begin{multline*} \bigcap \Big({\mathcal O}_{v_M^p} \cap K \;\Big|\;
K \subseteq M \textrm{ Galois}, \,[M:K]=n,\,M \neq M(p), \,
M \textrm{ not }p\textrm{-henselian,}\\ 
\zeta_p \in M \textrm{ if }
\mathrm{char}(M) \neq p,\,
v_M^p \in H_1^p(M) \Big).
\end{multline*}
\end{proof}

As an immediate consequence, we have the following
\begin{Cor}
Let $p$ be a prime and let $K$ be a field such that $p \mid \#G_K$ and that
$K$ is virtually not $p$-henselian. If 
$p=2$, assume that $K$ is not Euclidean. Then the power series
valuation is $\emptyset$-definable on $K((\Gamma))$, for any ordered abelian group 
$\Gamma$. 
\end{Cor}

Combining Theorem \ref{def1} with Lemma \ref{PAC}, we get:
\begin{Cor}
Let $(K,v)$ be a henselian valued field such that $Kv$ is PAC and
not separably closed.
Then $v$ is $\emptyset$-definable.
\end{Cor}

Another application of Theorem \ref{def1} are henselian valued fields with 
simple residue fields. We call a field \emph{simple} if $\mathrm{Th}(K)$ 
is simple in the sense of Shelah (see \cite{Wa00} for some background on 
simplicity). 
In a simple theory, no  
orderings with infinite chains are interpretable.
Thus, no simple field admits a definable 
valuation. Hence, by Theorem \ref{pdef}, 
simple fields cannot be $p$-henselian for any
prime $p$. As all Galois extensions of a simple field are interpretable in $K$
and thus again simple, any non-separably closed simple field $K$ is
not virtually $p$-henselian for any $p$ with $p \mid \#G_K$.
Thus, we get the following
\begin{Cor}
Let $(K,v)$ be a henselian valued field such that $Kv$ is simple and not
separably closed. Then $v$ is $\emptyset$-definable.
\end{Cor}

\subsection*{Real closed and separably closed residue fields}
In all our definitions of henselian valuations we showed so far that a given
henselian valuation $v$ on a field $K$ 
coincided with both the canonical henselian
valuation $v_K$ and the canonical $p$-henselian valuation
$v_K^p$ for some prime $p$. However, it can happen that some
 $v_K^p$ is henselian, but a proper coarsening of a given 
henselian valuation $v$.
In this case $v_K^p$ is again henselian and $\emptyset$-definable.
An example for this are henselian valued fields with 
separably closed residue field:
\begin{Thm} \label{H2}
Let $K$ be a field which is not separably closed. 
Assume that $K$ is henselian with respect
to a valuation with separably closed residue field.
Then $K$ admits a non-trivial $\emptyset$-definable henselian valuation.
\end{Thm}
\begin{proof}
We show first that $G_K$ is pro-soluble. If $K$ is henselian with
respect to a valuation with separably closed residue field, then
$v_K$ has also separably closed residue field. Let $w$ be the prolongation
of $v_K$ to $K^{sep}$. Recall that there is an exact sequence
$$ I_{w} \longrightarrow G_K \longrightarrow G_{Kv_K}$$
where $I_w$ denotes the inertia group of $w$ over $K$ (see 
\cite{EP05}, Theorem 5.2.7). Hence, as $I_w$ is pro-soluble 
(see \cite{EP05}, Lemma 5.3.2), so is $G_K$. 

Thus, there is some prime $p$ with $K \neq K(p)$.
But now $v_K^p$ is indeed a (not necessarily proper) 
coarsening of $v_K$: Otherwise, the definition of 
$v_K^p$ would imply $Kv_K \neq Kv_K(p)$.
If $K$ contains a
primitive $p$th root of unity or $\mathrm{char}(K)=p$, then $v_K^p$ is 
$\emptyset$-definable and henselian. Else, we consider the
$\emptyset$-definable extension $K(\zeta_p)$.
Then the canonical henselian valuation on $K(\zeta_p)$
still has separably closed residue field, therefore $v_{K(\zeta_p)}^p|_K$
gives a $\emptyset$-definable henselian valuation on $K$.
\end{proof}

\begin{Cor} \label{real}
Let $K$ be a field and assume that $K$ is not real closed. 
If $K$ is henselian with respect
to a valuation with real closed residue field,
then $K$ admits a non-trivial $\emptyset$-definable henselian valuation.
\end{Cor}
\begin{proof} 
If $(K,v)$ is henselian and $Kv$ is real closed, consider the unique
prolongation $w$ of $v$ to $L=K(i)$. The residue field $Lw$ is separably closed,
so L admits a $\emptyset$-definable henselian valuation by Theorem \ref{H2}.
As $v$ is the restriction of $w$ to $K$, $v$ is also $\emptyset$-definable on
$K$.
\end{proof}

\subsection{Henselian fields with non-universal absolute Galois groups}
In this section, we will give a Galois-theoretic condition
to ensure the existence of 
a non-trivial 
$\emptyset$-definable henselian valuation on a henselian field.

The following 
group-theoretic definition is taken from 
\cite{NS07}.
\begin{Def} 
Let $G$ be a profinite group.
We say that $G$ is \emph{universal} if every finite group occurs as a 
continuous subquotient of $G$.
\end{Def}
Note that for a field $K$, $G_K$ is non-universal iff there is some
$n \in \mathbb{N}$ such that the symmetric
group $S_n$ does not occur as a Galois group over any finite Galois
extension of $K$ (and then no $S_m$ with $m\geq n$ will occur).
The connection between non-universal absolute Galois groups and
henselianity is given by
the following statement:
\begin{Thm}[\cite{Koe05}, Theorem 3.1] \label{prod}
Let $K$ be a field and let $L$ and $M$ be
algebraic extensions of $K$ which both carry
non-trivial
henselian valuations. Assume further that $G_L$ is non-trivial pro-$p$ and $G_M$
non-trivial pro-$q$ for primes $p<q$. Let $v$ and $w$ be (not necessarily
proper) coarsenings of the canonical
henselian valuations on $L$ and $M$ respectively, and, if $p=2$ and $Lv$
is real closed, assume $v$ to be the coarsest henselian valuation on $L$
with real closed residue field. Then either $G_K$ 
is universal or $v|_K$ and $w|_K$ are comparable
and the coarser valuation is henselian on $K$.
\end{Thm}

\begin{Ex} All of the following profinite groups are non-universal:
\begin{enumerate}
\item pro-abelian groups,
\item pro-nilpotent groups,
\item pro-soluble groups,
\item any group $G$ such that $p \nmid \#G$ for some prime $p$.
\end{enumerate}
Non-abelian free profinite groups are of course universal, and so are
absolute Galois groups of hilbertian fields.
\end{Ex}

Now we can use Theorem \ref{prod} to deduce henselianity from
$p$- and $q$-henselianity:
\begin{Prop} \label{pq} 
Suppose $G_K$ is non-universal, 
and $K(p) \neq K \neq K(q)$ for two primes $p < q$.
In case $p=2$, assume further that $K$ is not Euclidean.
If $K$ is $p$- and $q$-henselian, 
then $K$ is henselian.
\end{Prop}
\begin{proof} 
Consider the henselization 
$L'$
(respectively $M'$) of $K$ with respect to the canonical
 $p$-henselian valuation $v_K^p$
(the canonical $q$-henselian valuation $v_K^q$) on $K$. 
Then define $L$ to be the fixed field of a $p$-Sylow subgroup of $G_{L'}$, and 
$M$ accordingly.
 
\emph{Claim:} $L$ is not separably closed.

\emph{Proof of Claim:} We need to show that $L'$ is not $p$-closed.
But if $\alpha \in K(p)$ has degree $p^n$ over $K$, then
-- as $v_K^p$ is $p$-henselian -- $\alpha$ is moved by some
element of $D(K(p)/K)$. As decomposition groups behave well in towers, we get
$\alpha \notin L.$ 

In case $p=2$, the same argument shows that $L$ is also not real closed.
Since $L$ is $p$-henselian and $G_L$ is pro-$p$, $L$ is also henselian, 
and likewise is $M$.
Now we consider the canonical henselian valuations $v_L$ on $L$
and the canonical henselian valuation $v_M$ on $M$.
If $p=2$ and $Lv_L$ is real closed, we replace $v_L$ 
by the coarsest henselian valuation on $L$ with 
real closed residue field. As $L$ is not real closed, this is again a 
non-trivial henselian valuation.

By Theorem \ref{prod}, the restrictions $v_L|_{K}$ and $v_M|_{K}$ are comparable
and the coarser one is henselian. As $L$ and $M$ are algebraic extensions
of $K$, none of the restrictions is trivial.
Hence $K$ is henselian.
\end{proof} 

\begin{Prop} \label{def2}
Let $G_K$ be non-universal.
Assume that there are
two primes $ p < q$ with $p,q \mid \# G_K$ and such that
$K(p) \neq K \neq K(q)$ holds. 
If $K$ is henselian, then $K$ is henselian
with respect to a non-trivial $\emptyset$-definable valuation. 
\end{Prop}
\begin{proof}
As long as we define a coarsening of $v_K$ without parameters, we
may assume that $\zeta_p, \zeta_q \in K$
if $\mathrm{char}(K)\neq p$ or $q$ respectively: 
The only special case is when $p=2$ and $K$ is Euclidean and $G_{K(i)}$ is 
pro-$q$. Then $K(i)$ already contains $\zeta_q$ and thus 
$v_{K(i)^q} = v_{K(i)}$ is a non-trivial 
$\emptyset$-definable henselian valuation 
on $K(i)$. In this case, its restriction to $K$ is
non-trivial $\emptyset$-definable henselian valuation on $K$. 

So now assume $\zeta_p, \zeta_q \in K$. In particular,
in case $p=2$, $K$ is not formally real and so $Kv_K^2$ cannot be 
Euclidean. All
these extensions still have non-universal absolute Galois groups. 

As $K$ is henselian, it is in particular $p$- and $q$-henselian.
We consider the canonical $p$-henselian ($q$-henselian) valuation $v_K^p$ 
($v_K^q$
respectively) on $K$. 
If $v_K^p$ or $v_K^q$ is henselian, then we have found
a $\emptyset$-definable henselian valuation. 

But this must always be the case: Assume that 
neither $v_K^p$ nor $v_K^q$ is henselian.
Then $v_K$ is a proper coarsening of $v_K^p$, 
and thus $Kv_K$ is $p$-henselian and
satisfies $Kv_K \neq Kv_K(p)$.
Similarly, $Kv_K$ is $q$-henselian and $Kv_K \neq Kv_K(q)$ holds.
Therefore, by Proposition 
\ref{pq}, $Kv_K$ is henselian. This contradicts the definition
of $v_K$.
\end{proof}

We can now prove our main result on henselian fields with non-universal absolute
Galois group.
\begin{Thm} \label{can} \label{non-universal}
Let $K$ be henselian, and assume that $G_K$ is non-universal. 
If $K$ is neither separably nor real closed, then
$K$ admits a $\emptyset$-definable henselian valuation.
If $Kv_K$ is neither separably nor real closed, then $v_K$ is
$\emptyset$-definable.
\end{Thm}
\begin{proof} By assumption, $K$ is neither separably nor real closed.
If $K$ is henselian and $Kv_K$ is separably closed (respectively real closed), 
then
$K$ admits a $\emptyset$-definable henselian valuation by Theorem \ref{H2}
(respectively Corollary \ref{real}). Thus, we may assume from now on that
$Kv_K$ is neither separably nor real closed.

In this case, 
$v_K$ is the finest henselian valuation
on $K$ and thus $Kv_K$ is not henselian.
Furthermore, there is some prime $p$ with
$p \mid \#G_{Kv_K}$. Assume first that $G_{Kv_K}$ is pro-$p$, 
then it follows that $Kv_K\neq Kv_K(p)$ and thus $K\neq K(p)$ 
(Proposition \ref{pex}). 
In particular,
$v_K$ must be a coarsening of $v_K^p$. But if $v_K$ was a proper coarsening
of $v_K^p$, then $Kv_K$ would be $p$-henselian and hence  -- as
$G_{Kv_K}$ is pro-$p$ -- henselian or real
closed.
Since we have assumed that $Kv_K$ is neither real closed nor henselian, 
we get $v_K = v_K^p$. 
As in previous proofs (see for example the proof of \ref{def2}), 
we may assume $\zeta_p \in K$ if
$\mathrm{char}(K)\neq p$, so $v_K$ is $\emptyset$-definable.

Now consider the case that there are (at least) two primes $p < q$ with
$p,q \mid \#G_{Kv_K}$. Thus, also $p,q \mid \# G_K$ holds.
If $Kv_K(p)\neq Kv_K \neq Kv_K(q)$, then 
-- using Proposition \ref{pex} once more -- we have
$K(p)\neq K \neq K(q)$. By the proof of Proposition \ref{def2}, 
one of $v_K^p$ and $v_K^q$ is henselian.
Say $v_K^p$ is henselian, then we get $v_K \subset v_K^p$ by Observation
\ref{coarse}.
But $v_K$ is also a coarsening of $v_K^p$, as $Kv_K \neq Kv_K(p)$. 
Thus, we conclude
$v_K=v_K^p$, and hence $v_K$ is again $\emptyset$-definable. 

Finally, if there are two primes $p,q \mid G_{Kv_K}$, but $Kv_K = Kv_K(p)$
or $Kv_K = Kv_K(q)$, we want to consider finite Galois extensions
$L$ of $Kv_K$ with $L(p)\neq L \neq L(q)$. 
Let $M$ be a finite Galois
extension of $K$, and let $w$ be the unique prolongation of $v_K$ to $M$.
Note that $G_M$ is again non-universal and, as $Mw$ is still neither
separably nor real closed, $w=v_M$ holds. 
If $Mw(p)\neq Mw \neq Mw(q)$, then $w$ is $\emptyset$-definable on
$M$ by $v_M^p$ or $v_M^q$ as above. Say $w = v_M^p$. 
As $w$ is in particular
$q$-henselian and $Mw \neq Mw(q)$, we get $w \supset v_m^q$.  
Thus, in any case the finest common coarsening
of $v_M^p$ and $v_M^q$ is equal to the coarser one of the two and furthermore
$\emptyset$-definable and henselian.

Now we fix an integer 
$n$ such that there is a Galois extension $M$ of $K$ (containing
$\zeta_p$ and $\zeta_q$ if necessary) such that
$Mw(p)\neq Mw \neq Mw(q)$.
Just like in the proof of \ref{def1}, we get a parameter-free definition of $v$
by
\begin{multline*} \bigcap \big(({\mathcal O}_{v_M^p}\cdot 
{\mathcal O}_{v_M^q})
\cap K \;\big|\;
K \subseteq M \textrm{ Galois}, \,[M:K]=n,\,M(p) \neq M \neq M(q),\\ 
\zeta_p \in M \textrm{ if }
\mathrm{char}(M) \neq p,\, \zeta_q \in M \textrm{ if }
\mathrm{char}(M) \neq q,\, 
v_M^p \in H_1^p(M), \, v_M^q \in H_1^q(M) \big).
\end{multline*}
\end{proof}

\begin{Rem}
In fact, it suffices to assume for the proof of the above theorem that
$K$ is $t$-henselian rather than henselian. This is a generalization 
of henselianity introduced in \cite{PZ78}. Like henselianity,
$t$-henselianity goes up to finite extensions and implies $p$-hense\-lianity
for any prime $p$. These are the only properties of henselianity needed in 
the proof.
In particular, we 
get that any field with a non-universal absolute Galois group which
is elementarily
equivalent to a henselian field is in fact henselian itself (since a
$\emptyset$-definable henselian valuation gives rise to a non-trivial henselian
valuation on any field with the same elementary theory). Thus, 
for any field with a non-universal absolute Galois group, henselianity is  
an elementary property in $\mathcal{L}_\textrm{ring}$. 
\end{Rem}

Our Galois-theoretic condition is moreover 
also a condition on the residue field.
\begin{Obs} \label{nu}
Let $(K,v)$ be a henselian valued field. Then $G_K$ is non-universal iff
$G_{Kv}$ is non-universal.
\end{Obs}
\begin{proof} Recall the exact sequence
$$I_v \longrightarrow G_{K} \longrightarrow G_{Kv}.$$
If $G_K$ is non-universal, then some finite group does not appear as 
a Galois group over any finite extension of $K$, and hence
the same holds for $Kv$.\\
On the other hand, if $G_{Kv}$ is non-universal, there is some $n_0 \in 
\mathbb{N}$
such that neither $S_n$ nor $A_n$ 
(for $n \geq n_0$) occur as a subquotients of $G_{Kv}$.
As $I_v$ is soluble, $S_n$ (for $n \geq \max\{5,n_0\}$) is not
a subquotient of $G_K$, either.
\end{proof}

In particular, we can use the observation to define a range of 
power series valuations.
\begin{Cor}
Let $K$ be a field with $G_K$ non-universal. Let
$\Gamma$ be
 a non-trivial 
ordered abelian group, and assume that
$\Gamma$ is non-divisible in case that $K$ is separably or real closed.
Then
there is a $\emptyset$-definable non-trivial
henselian valuation on $K((\Gamma))$. 
If $K$ is not henselian and neither
separably nor real closed, then the power 
series valuation is definable.
\end{Cor}
\begin{proof} The first statement is 
immediate from the previous Observation and Theorem 
\ref{non-universal}. The second also 
follows from Theorem \ref{can}: If $K$ is not 
henselian, then the power series valuation is exactly the canonical henselian
valuation.
\end{proof}

One example of fields with non-universal absolute Galois group are NIP fields
of positive characteristic. We call a field NIP if $\mathrm{Th}(K)$ is NIP
in the sense of Shelah (see \cite{Adl08} for some background on NIP theories).
In \cite{KSW11} (Corollary 4.5), the authors show that if $K$ is an 
infinite NIP field
of characteristic $p >0$, then $p \nmid \#G_K$. Thus, we get the following
\begin{Cor}
Let $(K,v)$ be a non-trivially henselian valued field, $K$
not separably closed. If
\begin{itemize} 
\item $K$ is NIP and $\mathrm{char}(K) > 0$, or
\item $Kv$ is NIP and $\mathrm{char}(Kv) > 0$,
\end{itemize} 
then $K$ admits a 
non-trivial
$\emptyset$-definable henselian valuation.
\end{Cor}
\begin{proof} The first statement follows from Theorem \ref{non-universal}. 
The second statement
is now a consequence of Observation \ref{nu}.
\end{proof}

\bibliographystyle{alpha}
\bibliography{franzi}
\end{document}